\documentclass[10pt]{amsart}
\usepackage[english]{babel}
\usepackage[T1]{fontenc}
\usepackage[latin1]{inputenc}
\usepackage{amsmath,amssymb,amsfonts}
\usepackage{amsthm}
\usepackage{dsfont}
\newcommand{\CC}{\mathbb{C}}
\newcommand{\KP}{\mathbb{KP}}
\newcommand{\Tr}{\mathrm{Tr}}
\newcommand{\cop}{\Delta}
\newcommand{\cou}{\varepsilon}
\newcommand{\hw}{\int_{\KP}\!\!}
\newcommand{\esp}{\mathbb{E}}
\newcommand{\Zn}{\mathbb{Z}_n}
\newcommand{\hwn}{\int_{\KP_n}\!\!}
\newcommand{\uD}{\mathbb{D}}
\newcommand{\uT}{\mathbb{T}}
\newcommand{\indic}{\mathds{1}}

\newcommand{\Xal}[2]{\mathfrak{X}_{#1,#2}}
\newcommand{\Xuv}[2]{\mathfrak{X}^{#1,#2}}
\DeclareMathOperator{\dplus}{\dot{+}}
\DeclareMathOperator{\alg}{*-alg}
\DeclareMathOperator{\id}{id}
\DeclareMathOperator{\LinSpan}{span}

\newtheorem{theorem}{Theorem}[section]

\newtheorem{lemma}{Lemma}[section]
\newtheorem{proposition}{Proposition}[section]
\newtheorem{definition}{Definition}[section]

\newtheorem{remark}{Remark}[section]

\begin{document}

\markboth{Isabelle Baraquin}
{Trace of Powers of Representations of Finite Quantum Groups}

%
%

\title{TRACE OF POWERS OF REPRESENTATIONS OF FINITE QUANTUM GROUPS}

\author{\footnotesize ISABELLE BARAQUIN}

\address{Laboratoire de math\'ematiques de Besan\c{c}on, UMR CNRS 6623\\Universit\'e Bourgogne Franche-Comt\'e\\
16 route de Gray\\
25030 Besan\c{c}on cedex\\
France}
\email{isabelle.baraquin@univ-fcomte.fr}

\begin{abstract}
In this paper we study (asymptotic) properties of the $*$-distribution of irreducible characters of finite quantum groups. We proceed in two steps, first examining the representation theory to determine irreducible representations and their powers, then we study the $*$-distribution of their trace with respect to the Haar measure. For the Sekine family we look at the asymptotic distribution (as the dimension of the algebra goes to infinity).
\end{abstract}

\maketitle


\section{Introduction}

In \cite{shah} and then in \cite{evans}, Diaconis, Shahshahani and Evans show that the traces of powers of a random unitary (respectively orthogonal) matrix behave asymptotically like independent complex (resp. real) Gaussian random variables. Later, Banica, Curran and Speicher investigate the case of easy quantum groups in \cite{BCS2011, BCS2012}, and obtain similar results in the context of free probability. 

Compact quantum groups are a generalization of classical compact groups. They were introduced in order to extend the Pontryagin duality. Despite of their name of group, we are studying algebras endowed with an additional structure. For a more complete presentation, the reader can look at \cite{MvD} or \cite{NeshveyevTuset}.

\begin{definition}[Woronowicz]
A compact quantum group is a pair  $\mathbb{G} = (\mathcal{A}, \cop)$, where $\mathcal{A}$ is a unital $C^*$-algebra (eventually noncommutative) and $\cop \colon \mathcal{A} \to \mathcal{A}\otimes \mathcal{A}$ is a unital $*$-homomorphism, called coproduct, such that it is coassociative, i.e.
\[\left( \cop \otimes \id_{\mathcal{A}} \right) \circ \cop = \left( \id_{\mathcal{A}} \otimes \cop \right) \circ \cop\]
as $*$-homomorphisms from $\mathcal{A}$ to $\mathcal{A} \otimes \mathcal{A} \otimes \mathcal{A}$, and such that it verifies the density in $\mathcal{A} \otimes \mathcal{A}$ of the two algebras $(1_{\mathcal{A}} \otimes \mathcal{A}) \cop(\mathcal{A})  = \LinSpan\left\{(1_{\mathcal{A}} \otimes a)\cop(b) , a, b \in \mathcal{A}\right\}$ and $(\mathcal{A} \otimes 1_{\mathcal{A}})\cop(\mathcal{A}) = \LinSpan\left\{(a \otimes 1_{\mathcal{A}})\cop(b) , a, b \in \mathcal{A}\right\}$.
\end{definition}

We will say that the quantum group is cocommutative if the coproduct is symmetric, i.e.\ if $\sigma \circ \cop = \cop$ where $\sigma(a \otimes b) = b \otimes a$, for all $a, b \in \mathcal{A}$. Note that we consider unital, associative and involutive algebras over the field of complex numbers and that the tensor products of algebras are algebraic tensor products over $\CC$.

Note that we can define a quantum group thanks to a classical compact group $(G,\cdot)$. The set of complex-valued continuous functions on $G$, denoted $C(G)$, is endowed with a structure of (commutative) unital $C^*$-algebra. Identifying the tensor product $C(G) \otimes C(G)$ with $C(G \times G)$, it is easy to check that the application $\cop \colon C(G) \to C(G \times G)$, defined by $\cop(f)(s,t) = f(s\cdot t)$ satisfies the coassociativity relation and the density properties, called quantum cancellation rules.

Following this example, we also denote the algebra $\mathcal{A}$ of the quantum group $\mathbb{G} = (\mathcal{A}, \cop)$ by $C(\mathbb{G})$, and a compact quantum group is called finite when the algebra is finite dimensional.

In this article, we will look at finite quantum groups. They were introduced in the sixties as examples of Hopf-von Neumann algebras to recover symmetry in duality for non abelian locally compact groups. The eight-dimensional Kac-Paljutkin quantum group $\KP$, introduced in \cite{KPtrad} is the smallest non-trivial example, in the sense it is neither commutative nor cocommutative. In $1996$, Sekine defines a new family of finite quantum groups \cite{sekine}, of dimension $2n^2$ for all $n \geq 2$.

To study a group, it is sometimes useful to look at its action on a complex Hilbert space $\mathcal{H}$. This is a representation of the group $\pi \colon G \to B(\mathcal{H})$, where $B(\mathcal{H})$ denotes the space of bounded operators on $\mathcal{H}$. When the group is compact, we can see $\pi$ as an element of $B(\mathcal{H}) \otimes C(G)$. In this framework, the property
\[\forall s, t \in G, \ \pi(s \cdot t) = \pi(s) \pi(t)\]
becomes $(\id_{B(\mathcal{H})} \otimes \cop)\pi = J_{12}(\pi) J_{13}(\pi)$, where for every $a \otimes f \in B(\mathcal{H}) \otimes C(G)$, $J_{12}(a \otimes f) = a \otimes f \otimes \indic_G$ and $J_{13}(a \otimes f) = a \otimes \indic_G \otimes f$ in $B(\mathcal{H}) \otimes C(G) \otimes C(G)$ and $\indic_G$ denotes the constant function $1$ on $G$, the unit in $C(G)$.
\begin{definition}
Let $\mathbb{G} = (C(\mathbb{G}), \cop)$ be a compact quantum group. A corepresentation of the algebra $C(\mathbb{G})$, also called a representation of the quantum group $\mathbb{G}$, is an invertible element $u$ of $B(\mathcal{H}) \otimes C(\mathbb{G})$, for some complex Hilbert space $\mathcal{H}$, such that $(\id_{B(\mathcal{H})} \otimes \cop) u = J_{12}(u)J_{13}(u)$ in $B(\mathcal{H}) \otimes C(\mathbb{G}) \otimes C(\mathbb{G})$.
\end{definition}

We will say that a representation $u \in B(\mathcal{H}) \otimes C(\mathbb{G})$ of the compact quantum group $\mathbb{G}$ is unitary if $u$ is a unitary element of $B(\mathcal{H}) \otimes C(\mathbb{G})$, finite dimensional if the Hilbert space $\mathcal{H}$ is finite dimensional.

Thus, once a basis of $\mathcal{H}$ is fixed, we can see a finite dimensional representation as an element of $\mathcal{M}_n(\CC) \otimes C(\mathbb{G}) \simeq \mathcal{M}_n\left(C(\mathbb{G})\right)$, i.e.\ a matrix $u = (u_{ij})_{1 \leq i,j \leq n}$ with coefficients in the algebra of the quantum group satisfying $\cop(u_{ij}) = \sum\limits_{k} u_{ik} \otimes u_{kj}$.

\begin{definition}
Assume $u \in B(\mathcal{H}_u) \otimes C(\mathbb{G})$ and $v \in B(\mathcal{H}_v) \otimes C(\mathbb{G})$ are representations of the compact quantum group $\mathbb{G}$. We say that an operator $T \colon \mathcal{H}_u \to \mathcal{H}_v$ intertwines $u$ and $v$ if
\[(T \otimes 1_{C(\mathbb{G})}) u = v (T \otimes 1_{C(\mathbb{G})}) \text{ .}\]
If such an operator exists, $u$ and $v$ are called equivalent. If moreover $T$ is unitary, they are called unitarily equivalent.

A representation $u$ is called irreducible if there is no intertwiners between $u$ and $u$ itself except scalar multiples of $\id_{\mathcal{H}}$.
\end{definition}

In particular irreducible representations of compact quantum groups are finite dimensional. An irreducible representation $u$ of a compact quantum group $\mathbb{G}$ will be called fundamental if its coefficients $u_{ij}$ generate the algebra $C(\mathbb{G})$.

Up to equivalence, we can consider only unitary irreducible (finite dimensional) representations. With these objects, we will use the following notation:
\begin{definition}
Let $M$ be a matrix whose coefficients $M_{i,j}$ are elements of an algebra $\mathcal{A}$ (eventually noncommutative). The trace of $M$, denoted $\chi(M)$, is the sum of all its diagonal elements, that is $\chi(M) = \sum\limits_{i} M_{ii}$ in $\mathcal{A}$.
\end{definition}

Moreover, every classical compact group can be equipped with a Haar measure. If we look at integration of continuous functions with respect to this measure, we obtain a state on the group algebra. The identity
\[\int_G f(g \cdot s)\,\mathrm{d}\lambda(s) = \int_G f(s)\,\mathrm{d}\lambda(s) = \int_G f(s \cdot g)\,\mathrm{d}\lambda(s)\]
becomes $(\lambda \otimes \id_{C(G)}) \circ \cop = \lambda(\cdot) \indic_G = (\id_{C(G)} \otimes \lambda) \circ \cop$, with the coproduct defined above, where $\lambda$ denotes the integration with respect to the Haar measure. This leads us to the following definition:
\begin{definition}
A Haar state $h$ on a compact quantum group $\mathbb{G} = (\mathcal{A}, \cop)$ is a state on $\mathcal{A}$ such that
\[(h \otimes \id_{\mathcal{A}})\circ \cop = h(\cdot) 1_{\mathcal{A}} = (\id_{\mathcal{A}} \otimes h)\circ \cop\text{ .}\]
\end{definition}

Note that every compact quantum group admits a unique Haar state. We will use it to determine $*$-distribution of elements of the algebra, in the following sense. Let us denote by $\CC\langle X, X^* \rangle$ the set of polynomials with two (noncommutative) variables $X$ and $X^*$.

\begin{definition}
Let $a$ be an element of the compact quantum group $\mathbb{G} = (\mathcal{A}, \cop)$, endowed with the Haar state $h$. The $*$-distribution of $a$ is the linear functional $\mu_a \colon \CC\langle X, X^* \rangle \to \CC$ given by
\[\mu_a\left(X^{\varepsilon_1} \ldots X^{\varepsilon_n}\right) = h(a^{\varepsilon_1} \ldots a^{\varepsilon_n})\]
for all natural number $n$ and every choice of $\varepsilon_1, \ldots, \varepsilon_n \in \{\emptyset, *\}$, where $a^{\varepsilon_1} \ldots a^{\varepsilon_n}$ denotes the corresponding product of $a = a^{\emptyset}$ and $a^*$ in $\mathcal{A}$. The $h(a^{\varepsilon_1} \ldots a^{\varepsilon_n})$'s are called the moments of $a$.
\end{definition}

We will compare these $*$-distributions with classical probability distributions. For instance, the Dirac mass probability distribution $\delta_x$ in some real number $x$, the arcsine distribution $\mu_{arc(x,y)}$ on the interval $[x\;;\;y]$, or the uniform distribution $\mathcal{U}(x\uT)$ on a multiple of the complex unit circle $\uT$.

We will also look at joint distribution of several elements, defined in a similar way. In this article, we say that the elements are independent when the joint distribution correspond to the joint distribution of classical independent random variables. We consider asymptotic distribution as convergence in moments.

This work is separated into two parts. The first section is devoted to the study of the Kac-Paljutkin quantum group $\KP$. After recalling the definition, we will determine its five irreducible representations. Finally, we determine the $*$-distribution of the trace of powers of the fundamental representation in Theorem \ref{ThDist} and study their independence in Theorem \ref{ThInd}.

In the second section, we work with the family of Sekine quantum groups. After recalling the definition, we give the representation theory. We also study the character space and a commutative subalgebra. Finally, we determine the asymptotic $*$-distribution of the trace of powers of two-dimensional representations in Theorems \ref{ThOdd} and \ref{ThEven}.

\section{Kac-Paljutkin Finite Quantum Group $\KP$}

\subsection{Definition}

We will follow the definition of \cite{uweKP}, but for convenience of the reader, we recall here the notations.

Consider the multi-matrix algebra $\mathcal{A} = \CC \oplus \CC \oplus \CC \oplus \CC \oplus \mathcal{M}_2(\CC)$ together with usual multiplication and involution. This is an eight-dimensional algebra, with the canonical basis
\begin{eqnarray*}
e_1 = 1 \dplus 0 \dplus 0 \dplus 0 \dplus \begin{pmatrix} 0 & 0 \\ 0 & 0 \end{pmatrix} & \ \ \ \ \ \ & E_{11} = 0 \dplus 0 \dplus 0 \dplus 0 \dplus \begin{pmatrix} 1 & 0 \\ 0 & 0 \end{pmatrix}\\
e_2 = 0 \dplus 1 \dplus 0 \dplus 0 \dplus \begin{pmatrix} 0 & 0 \\ 0 & 0 \end{pmatrix} & \ \ \ \ \ \ & E_{12} = 0 \dplus 0 \dplus 0 \dplus 0 \dplus \begin{pmatrix} 0 & 1 \\ 0 & 0 \end{pmatrix}\\
e_3 = 0 \dplus 0 \dplus 1 \dplus 0 \dplus \begin{pmatrix} 0 & 0 \\ 0 & 0 \end{pmatrix} & \ \ \ \ \ \ & E_{21} = 0 \dplus 0 \dplus 0 \dplus 0 \dplus \begin{pmatrix} 0 & 0 \\ 1 & 0 \end{pmatrix}\\
e_4 = 0 \dplus 0 \dplus 0 \dplus 1 \dplus \begin{pmatrix} 0 & 0 \\ 0 & 0 \end{pmatrix} & \ \ \ \ \ \ & E_{22} = 0 \dplus 0 \dplus 0 \dplus 0 \dplus \begin{pmatrix} 0 & 0 \\ 0 & 1 \end{pmatrix}
\end{eqnarray*}
where $\dplus$ is defined in a natural way to designate elements in the direct sum. The unit is naturally $\mathds{1} = 1 \dplus 1 \dplus 1 \dplus 1 \dplus \begin{pmatrix} 1 & 0 \\ 0 & 1 \end{pmatrix} = e_1 + e_2 + e_3 + e_4 + E_{11} + E_{22}$.

The following defines the coproduct, where $\imath$ is the imaginary unit:
\begin{align*}
\cop(e_1) = e_1 &\otimes e_1 + e_2 \otimes e_2 + e_3 \otimes e_3 + e_4 \otimes e_4\\
& + \frac{1}{2} E_{11} \otimes E_{11} + \frac{1}{2} E_{12} \otimes E_{12} + \frac{1}{2} E_{21} \otimes E_{21} + \frac{1}{2} E_{22} \otimes E_{22}\\
\cop(e_2) = e_1 &\otimes e_2 + e_2 \otimes e_1 + e_3 \otimes e_4 + e_4 \otimes e_3\\
& + \frac{1}{2} E_{11} \otimes E_{22} + \frac{1}{2} E_{22} \otimes E_{11} - \frac{\imath}{2} E_{12} \otimes E_{21} + \frac{\imath}{2} E_{21} \otimes E_{12}\\
\cop(e_3) = e_1 &\otimes e_3 + e_3 \otimes e_1 + e_2 \otimes e_4 + e_4 \otimes e_2\\
& + \frac{1}{2} E_{11} \otimes E_{22} + \frac{1}{2} E_{22} \otimes E_{11} + \frac{\imath}{2} E_{12} \otimes E_{21} - \frac{\imath}{2} E_{21} \otimes E_{12}\\
\cop(e_4) = e_1 &\otimes e_4 + e_4 \otimes e_1 + e_2 \otimes e_3 + e_3 \otimes e_2\\
& + \frac{1}{2} E_{11} \otimes E_{11} + \frac{1}{2} E_{22} \otimes E_{22} - \frac{1}{2} E_{12} \otimes E_{12} - \frac{1}{2} E_{21} \otimes E_{21}\\
\cop(E_{11}) = e_1 &\otimes E_{11} + E_{11} \otimes e_1 + e_2 \otimes E_{22} + E_{22} \otimes e_2\\
& + e_3 \otimes E_{22} + E_{22} \otimes e_3 + e_4 \otimes E_{11} + E_{11} \otimes e_4\\
\cop(E_{12}) = e_1 &\otimes E_{12} + E_{12} \otimes e_1 + \imath e_2 \otimes E_{21} - \imath E_{21} \otimes e_2\\
& - \imath e_3 \otimes E_{21} + \imath E_{21} \otimes e_3 - e_4 \otimes E_{12} - E_{12} \otimes e_4\\
\cop(E_{21}) = e_1 &\otimes E_{21} + E_{21} \otimes e_1 - \imath e_2 \otimes E_{12} + \imath E_{12} \otimes e_2\\
& + \imath e_3 \otimes E_{12} - \imath E_{12} \otimes e_3 - e_4 \otimes E_{21} - E_{21} \otimes e_4\\
\cop(E_{22}) = e_1 &\otimes E_{22} + E_{22} \otimes e_1 + e_2 \otimes E_{11} + E_{11} \otimes e_2\\
& + e_3 \otimes E_{11} + E_{11} \otimes e_3 + e_4 \otimes E_{22} + E_{22} \otimes e_4
\end{align*}
the counit is given by $\cou \left(x_1 \dplus x_2 \dplus x_3 \dplus x_4 \dplus \begin{pmatrix}c_{11} & c_{12}\\c_{21} & c_{22}\end{pmatrix}\right) = x_1$ and the antipode is the transpose map, i.e.\ $S(e_i) = e_i$ and $S(E_{ij}) = E_{ji}$.

This defines a finite quantum group, denoted by $\KP = \left( \mathcal{A}, \cop\right)$. We shall also need its Haar state, denoted by $\hw$ :
\[\hw \left(x_1 \dplus x_2 \dplus x_3 \dplus x_4 \dplus \begin{pmatrix}c_{11} & c_{12}\\c_{21} & c_{22}\end{pmatrix}\right) = \frac{1}{8} \left( x_1 + x_2 + x_3 + x_4 + 2 (c_{11} + c_{22}) \right) \text{ .}\]

\subsection{The group of group-like elements}

A group-like element is a non-zero elements of $\mathcal{A}$ such that $\cop(g) = g \otimes g$.
Group-like elements satisfying $\cou(g) = 1$, i.e.\  $g_1 = 1$, correspond to one dimensional representations of $\KP$. Moreover, they form a group, with inverse given by the antipode. In particular, the unit $\indic$, called the trivial representation, satisfies these conditions.

By direct calculation, using the conditions above and the linear basis, the family of group-like elements of $\KP$ is a group isomorphic to $\mathbb{Z}_2 \times \mathbb{Z}_2$:
\begin{multline}
\left\{1 \dplus 1 \dplus 1 \dplus 1 \dplus \begin{pmatrix}1&0\\0&1\end{pmatrix}, \ 1 \dplus 1 \dplus 1 \dplus 1 \dplus \begin{pmatrix}-1&0\\0&-1\end{pmatrix}, \right.\\ \left. 1 \dplus -1 \dplus -1 \dplus 1 \dplus \begin{pmatrix}1&0\\0&-1\end{pmatrix}, \ 1 \dplus -1 \dplus -1 \dplus 1 \dplus \begin{pmatrix}-1&0\\0&1\end{pmatrix}\right\}\text{ .}
\label{grplike}
\end{multline}

\begin{remark} \label{rq1}
Note that in \cite{Izumi}, Izumi and Kosaki do not follow the same way to define $\KP$. They give group-like elements which look a little bit different from ours. The fact is that they do not use the same basis: in their notation, $z(a)$, $z(b)$ and $z(c)$ play respectively the role of our $e_4$, $e_2$ and $e_3$.
\end{remark}

\subsection{Matrix elements and fundamental representation}

Let us look at representations of $\KP$ with dimension at least $2$. We will determine matrix elements of representation of dimension $2$, i.e.\ elements $X_{(11)}$, $X_{(12)}$, $X_{(21)}$, $X_{(22)}$ of $\mathcal{A}$, such that $\cop(X_{(ij)}) = X_{(i1)} \otimes X_{(1j)} + X_{(i2)} \otimes X_{(2j)}$ and $\cou(X_{(ij)}) = \delta_{i,j}$ for all $i, j \in \{1,2\}$. These are matrix elements of two-dimensional representations of $\KP$.

\begin{proposition}
For all $a \in \{-1, 1\}$ and all $\lambda \in \mathbb{T}$, let us fix
\begin{align*}
\Xal{a}{\lambda} &= \begin{pmatrix}\mathfrak{X}_{(11)}&\mathfrak{X}_{(12)}\\\mathfrak{X}_{(21)}&\mathfrak{X}_{(22)}\end{pmatrix}\\
&= \begin{pmatrix}
1 \dplus a \dplus -a \dplus -1 \dplus \begin{pmatrix}0&0\\0&0\end{pmatrix} & 0 \dplus 0 \dplus 0 \dplus 0 \dplus \begin{pmatrix}0&\lambda\\\imath a \lambda&0\end{pmatrix}\\
0 \dplus 0 \dplus 0 \dplus 0 \dplus \begin{pmatrix}0&\bar{\lambda}\\-\imath a \bar{\lambda}&0\end{pmatrix} & 1 \dplus -a \dplus a \dplus -1 \dplus \begin{pmatrix}0&0\\0&0\end{pmatrix}
\end{pmatrix} \text{ .}
\end{align*}
Then $\Xal{a}{\lambda}$ is a fundamental representation of $\KP$, it means that its coefficients generate the algebra $\mathcal{A}$.
\end{proposition}

\begin{remark}
Remark \ref{rq1}, about \cite{Izumi}, holds again. Moreover, since all $\Xal{a}{\lambda}$ are unitary equivalent, Izumi and Kosaki fix $a = -1$ and $\lambda = e^{\imath \frac{\pi}{4}}$. 
\end{remark}

\begin{proof}
First of all, let us check that $\mathfrak{X}_{(11)}$ is a matrix element of a two-dimensional representation of $\KP$. The computation for $\mathfrak{X}_{(12)}$, $\mathfrak{X}_{(21)}$ and $\mathfrak{X}_{(22)}$ are similar.
We have on the first hand
\begin{align*}
\cop(\mathfrak{X}_{(11)}) = e_1 &\otimes e_1 + e_2 \otimes e_2 + e_3 \otimes e_3 + e_4 \otimes e_4\\
& - \left(e_1 \otimes e_4 + e_4 \otimes e_1 + e_2 \otimes e_3 + e_3 \otimes e_2\right)\\
& + a\left(e_1 \otimes e_2 + e_2 \otimes e_1 + e_3 \otimes e_4 + e_4 \otimes e_3 \right)\\
& - a\left(e_1 \otimes e_3 + e_3 \otimes e_1 + e_2 \otimes e_4 + e_4 \otimes e_2\right)\\
& + E_{12} \otimes E_{12} + E_{21} \otimes E_{21} + \imath a E_{21} \otimes E_{12} - \imath a E_{12} \otimes E_{21}
\end{align*}
and, on the other hand
\begin{align*}
\mathfrak{X}_{(11)} \otimes \mathfrak{X}_{(11)} + \mathfrak{X}_{(12)} &\otimes \mathfrak{X}_{(21)} = e_1 \otimes e_1 + a e_1 \otimes e_2 -a e_1 \otimes e_3 - e_1 \otimes e_4\\
& + a \left(e_2 \otimes e_1 + a e_2 \otimes e_2 -a e_2 \otimes e_3 - e_2 \otimes e_4\right)\\
& - a\left(e_3 \otimes e_1 + a e_3 \otimes e_2 - a e_3 \otimes e_3 - e_3 \otimes e_4 \right)\\
& - \left(e_4 \otimes e_1 + a e_4 \otimes e_2 - a e_4 \otimes e_3 - e_4 \otimes e_4\right)\\
& + |\lambda|^2 E_{12} \otimes E_{12} -\imath a |\lambda|^2 E_{12} \otimes E_{21}\\
& + \imath a |\lambda|^2 E_{21} \otimes E_{12} + a^2 |\lambda|^2 E_{21} \otimes E_{21} \text{ .}
\end{align*}
Hence, $\cop(\mathfrak{X}_{(11)}) = \mathfrak{X}_{(11)} \otimes \mathfrak{X}_{(11)} + \mathfrak{X}_{(12)} \otimes \mathfrak{X}_{(21)}$.

Moreover, we can show that $\Xal{a}{\lambda}$ and $\overline{\Xal{a}{\lambda}}$ are unitary matrices and that we have $\hw\,\chi(\Xal{a}{\lambda})^* \chi(\Xal{a}{\lambda}) = 1$. Hence $\Xal{a}{\lambda}$ defines a unitary irreducible representation of the finite quantum group $\KP$.

Finally, the family $\left\{\mathfrak{X}_{(11)}, \mathfrak{X}_{(12)}, \mathfrak{X}_{(21)}, \mathfrak{X}_{(22)} \right\}$ generates $\mathcal{A}$, since
\begin{align*}
e_1 &= \frac{1}{4} \left( \mathfrak{X}_{(11)}^2 + \mathfrak{X}_{(11)} \mathfrak{X}_{(22)} \right) + \frac{1}{4} \left(\mathfrak{X}_{(11)} + \mathfrak{X}_{(22)} \right)\\
e_4 &= \frac{1}{4} \left( \mathfrak{X}_{(11)}^2 + \mathfrak{X}_{(11)} \mathfrak{X}_{(22)} \right) - \frac{1}{4} \left(\mathfrak{X}_{(11)} + \mathfrak{X}_{(22)} \right)\\
e_2 &= \frac{1}{4} \left( \mathfrak{X}_{(11)}^2 - \mathfrak{X}_{(11)} \mathfrak{X}_{(22)} \right) + \frac{a}{4} \left(\mathfrak{X}_{(11)} - \mathfrak{X}_{(22)} \right)\\
e_3 &= \frac{1}{4} \left( \mathfrak{X}_{(11)}^2 - \mathfrak{X}_{(11)} \mathfrak{X}_{(22)} \right) - \frac{a}{4} \left(\mathfrak{X}_{(11)} - \mathfrak{X}_{(22)} \right)\\
E_{11} &=\frac{1}{2} \left(\mathfrak{X}_{(12)} \mathfrak{X}_{(12)}^* + \imath a \mathfrak{X}_{(12)} \mathfrak{X}_{(21)} \right)\\
E_{22} &=\frac{1}{2} \left(\mathfrak{X}_{(12)} \mathfrak{X}_{(12)}^* - \imath a \mathfrak{X}_{(12)} \mathfrak{X}_{(21)} \right)\\
E_{12} &=\frac{1}{2} \left( \bar{\lambda} \mathfrak{X}_{(12)} + \lambda \mathfrak{X}_{(21)} \right)\\
E_{21} &=- \frac{\imath a}{2} \left( \bar{\lambda} \mathfrak{X}_{(12)} - \lambda \mathfrak{X}_{(21)} \right)\text{ .}
\end{align*}
\end{proof}

\subsection{Powers of fundamental representation}

To extend the study of Diaconis and Shahshahani in $\KP$, we compute the traces of powers of the fundamental representation. Once we have these elements, we compute their $*$-distribution in the finite quantum group $\KP$ with respect to the Haar state, and identify them with classical probability distributions determined by moments of all order.

To do this, we have to define the adjoint $x^*$ of an element $x$ in $\mathcal{A}$. It is given by the following formula:
\[\left(x_1 \dplus x_2 \dplus x_3 \dplus x_4 \dplus \begin{pmatrix}c_{11} & c_{12}\\c_{21} & c_{22}\end{pmatrix}\right)^* = \overline{x_1} \dplus \overline{x_2} \dplus \overline{x_3} \dplus \overline{x_4} \dplus \begin{pmatrix}\overline{c_{11}} & \overline{c_{21}}\\\overline{c_{12}} & \overline{c_{22}}\end{pmatrix} \text{ .}\]

First, let us compute the powers of $\Xal{a}{\lambda}$. Let $\{e_{ij}\}_{1 \leq i,j \leq 2}$ denote the family of matrix unit in $\mathcal{M}_2(\CC)$. Then
\[\Xal{a}{\lambda} = \sum_{i,j = 1}^2 e_{ij} \otimes \mathfrak{X}_{(ij)}\]
as an element of $\mathcal{M}_2(\CC) \otimes \mathcal{A}$. The usual multiplication in the tensor product of algebra gives, going back to the matrix notation,
\[\left(\Xal{a}{\lambda}\right)^2 = \begin{pmatrix}
1 \dplus 1 \dplus 1 \dplus 1 \dplus \begin{pmatrix} - \imath a&0\\0&\imath a\end{pmatrix}& 0 \dplus 0 \dplus 0 \dplus 0 \dplus \begin{pmatrix} 0&0\\0&0\end{pmatrix}\\
0 \dplus 0 \dplus 0 \dplus 0 \dplus \begin{pmatrix} 0&0\\0&0\end{pmatrix} & 1 \dplus 1 \dplus 1 \dplus 1 \dplus \begin{pmatrix} \imath a&0\\0&- \imath a\end{pmatrix}
\end{pmatrix} \text{ .}\]

By similar calculations, using $\left(\Xal{a}{\lambda}\right)^{2n} = \left[\left(\Xal{a}{\lambda}\right)^{2}\right]^n$, for all non negative integers $n$, $\left(\Xal{a}{\lambda}\right)^{2n}$ is a diagonal matrix and
\begin{equation}
\label{Lm1}
\left(\Xal{a}{\lambda}\right)^{2n} = \begin{pmatrix}
1 \dplus 1 \dplus 1 \dplus 1 \dplus \begin{pmatrix} (- \imath a)^n\hspace{-5pt}&0\\0&\hspace{-5pt}(\imath a)^n\end{pmatrix}\hspace{-10pt}& 0 \dplus 0 \dplus 0 \dplus 0 \dplus \begin{pmatrix} 0&0\\0&0\end{pmatrix}\\
0 \dplus 0 \dplus 0 \dplus 0 \dplus \begin{pmatrix} 0&0\\0&0\end{pmatrix} & \hspace{-10pt}1 \dplus 1 \dplus 1 \dplus 1 \dplus \begin{pmatrix} (\imath a)^n\hspace{-5pt}&0\\0&\hspace{-5pt}(- \imath a)^n\end{pmatrix}
\end{pmatrix}
\end{equation}

\begin{lemma}
\label{Lm2}
For all non negative integer $n$, we have $\chi\left(\left(\Xal{a}{\lambda}\right)^{2n+1}\right) = \chi\left(\Xal{a}{\lambda}\right)$ and $\chi\left(\left(\Xal{a}{\lambda}\right)^{2n}\right) = 2 \dplus 2 \dplus 2 \dplus 2 \dplus \left((\imath a)^n + (-\imath a)^n\right) I_2$, where $I_2$ is the identity matrix in $\mathcal{M}_2(\CC)$.
\end{lemma}

\begin{proof}
The second relation comes from equation \eqref{Lm1} and the componentwise addition. The classical formula $\left(\Xal{a}{\lambda}\right)^{2n+1} = \left(\Xal{a}{\lambda}\right)^{2n} \Xal{a}{\lambda}$ leads to the following expression for $\left(\Xal{a}{\lambda}\right)^{2n+1}$:
\[\begin{pmatrix}
1 \dplus a \dplus -a \dplus -1 \dplus \begin{pmatrix}0&0\\0&0\end{pmatrix} &
0 \dplus 0 \dplus 0 \dplus 0 \dplus \begin{pmatrix}0&(-\imath a)^n\lambda\\(\imath a)^{n+1} \lambda&0\end{pmatrix}\\
0 \dplus 0 \dplus 0 \dplus 0 \dplus \begin{pmatrix}0&(\imath a)^n\bar{\lambda}\\(-\imath a)^{n+1} \bar{\lambda}&0\end{pmatrix} &
1 \dplus -a \dplus a \dplus -1 \dplus \begin{pmatrix}0&0\\0&0\end{pmatrix}
\end{pmatrix}\]
and the trace does not depends on $n$.
\end{proof}

Let $\esp[Z^n]$ denotes moment of order $n$ of a classical random variable $Z$. We obtain four different discrete probability distributions for $\chi\left(\left(\Xal{a}{\lambda}\right)^k\right)$, depending on the power $k$.

\begin{theorem}
\label{ThDist}
Let $k$ be a non negative integer. Let us denote by $\mu_0,\  \mu_1,\ \mu_2$ and $\mu_4$ the following distributions:
\[\mu_0 =\delta_2,\ \mu_1 = \frac{1}{8} \delta_{-2} + \frac{3}{4} \delta_0 + \frac{1}{8} \delta_2,\ \mu_2 = \frac{1}{2} \delta_0 + \frac{1}{2} \delta_2 \text{ and }\mu_4 =\frac{1}{2} \delta_{-2} + \frac{1}{2} \delta_2 \text{ .}\]

Then for all $a \in \{-1, 1\}$ and $\lambda \in \mathbb{T}$, $\chi\left((\Xal{a}{\lambda})^k\right)$ is self-adjoint and admits $\mu_{\kappa}$ as $*-$distribution, with
\[\kappa = \begin{cases}
1 & \text{ if } k \equiv 1 [2]\\
2 & \text{ if } k \equiv 2 [4]\\
4 & \text{ if } k\equiv 4 [8]\\
0 & \text{ if } k \equiv 0 [8]
\end{cases} \text{ .}\]

Moreover, we have
\[ \chi\left((\Xal{a}{\lambda})^k\right) = \begin{cases}
2 \dplus 0 \dplus 0 \dplus -2 \dplus 0_2 & \text{, if }k \equiv 1[2]\\
2 \dplus 2 \dplus 2 \dplus 2 \dplus 0_2 & \text{, if }k \equiv 2[4]\\
2 \dplus 2 \dplus 2 \dplus 2 \dplus -2I_2 & \text{, if }k \equiv 4[8]\\
2 \dplus 2 \dplus 2 \dplus 2 \dplus 2I_2 & \text{, if }k \equiv 0[8]\\
\end{cases} \text{ .}
\]
\end{theorem}

\begin{proof}
Assume that $k$ is odd, then $\chi\left(\left(\Xal{a}{\lambda}\right)^k\right) = \chi\left(\Xal{a}{\lambda}\right)$, by Lemma \ref{Lm2}, and, by definition, we have
\[\chi\left(\Xal{a}{\lambda}\right) = \mathfrak{X}_{(11)} + \mathfrak{X}_{(22)} = 2 \dplus 0 \dplus 0 \dplus -2 \dplus \begin{pmatrix}0&0\\0&0\end{pmatrix}\]
so, $\chi\left(\Xal{a}{\lambda}\right)^* = \chi\left(\Xal{a}{\lambda}\right)$, and for all non negative integer $n$, we obtain that
\begin{align*}
\hw \left(\chi\left(\Xal{a}{\lambda}\right)\right)^n &= \frac{1}{8} (2^n + 0^n + 0^n + (-2)^n + 2(0^n + 0^n))\\
&= \frac{(-2)^n}{8} + \frac{3}{4} \times 0^n + \frac{2^n}{8} = \esp[Z_1^n]
\end{align*}
where $Z_1$ is a $\mu_1$-distributed random variable.

Now, assume that $k$ is even. Then
\[\chi\left(\left(\Xal{a}{\lambda}\right)^k\right) = 2 \dplus 2 \dplus 2 \dplus 2 \dplus \begin{pmatrix}(\imath a)^\frac{k}{2} + (-\imath a)^\frac{k}{2} & 0\\0 & (\imath a)^\frac{k}{2} + (-\imath a)^\frac{k}{2}\end{pmatrix}\]
is self-adjoint and
\[\hw\left(\chi\left(\left(\Xal{a}{\lambda}\right)^k\right)\right)^n = \frac{1}{2}\left(2^n +\left((\imath a)^\frac{k}{2} + (-\imath a)^\frac{k}{2}\right)^n\right) \text{ .}\]
Let us note that $(\imath a)^2 = -1$, and $(\imath a)^4 = 1$. Hence, we obtain that the distribution of $\chi\left(\left(\Xal{a}{\lambda}\right)^k\right)$ equals $\mu_2$ if $k = 4p+2$, $\mu_4$ if $k = 8p+4$ and $\mu_0$ if $8$ divides $k$.
\end{proof}

\begin{remark}
Like in the classical case, we are able to express traces of powers of the fundamental representation as linear combinations of irreducible characters, it means one-dimensional representations, listed in \eqref{grplike}, and $\chi\left(\Xal{a}{\lambda}\right)$. Here we have $2 \dplus 2 \dplus 2 \dplus 2 \dplus 2I_2 = 2\; \mathds{1}$, $2 \dplus 2 \dplus 2 \dplus 2 \dplus -2I_2 = 2 \left(1 \dplus 1 \dplus 1 \dplus 1 \dplus -I_2\right)$ and $2 \dplus 2 \dplus 2 \dplus 2 \dplus 0_2 = 2 \left(1 \dplus 1 \dplus 1 \dplus 1 \dplus -I_2 + \mathds{1} \right)$. Let us note that this is not true in general for quantum groups. For instance, in the free orthogonal group $O_N^+$, $\chi(u^2)$ is not a linear combination of characters, where $u$ is the fundamental representation of $O_N^+$. For further examples, the reader can also look at the dual quantum group $\widehat{\KP_n}$ in subsection \ref{ssDual}.
\end{remark}

We are now able to study the independence relations between the four distinct variables obtained in the previous theorem. To do this, we work with joint cumulants. It vanishes whenever there is a random variable independent from the others. Let us denote by $\mathcal{P}_r$ the set of all partitions of $\{1, 2, \ldots, r\}$, and for $\pi \in \mathcal{P}_r$, $|\pi|$ is the number of its blocks.

\begin{theorem}
\label{ThInd}
For $i \in \{0,1,2,4\}$, let $Z_i$ be a $\mu_i$-distributed random variable such that $Z_0$ and $Z_1$ are independent from all the others.

Then, for all $a \in \{-1, 1\}$, $\lambda \in \mathbb{T}$, and $(k_1, \ldots, k_r) \in  \mathbb{N}^r$,
\begin{multline*}
\hw \chi\left((\Xal{a}{\lambda})^{k_1}\right) \ldots \chi\left((\Xal{a}{\lambda})^{k_r}\right) = \esp\left[Z_{m_1} \ldots Z_{m_r} \right]\\
= 2^{\#\{i, k_i = 0\}} \esp\left[ Z_1^{\#\{i, k_i = 1\}} \right] \esp\left[Z_2^{\#\{i, k_i = 2\}} Z_4^{\#\{i, k_i = 4\}} \right]
\end{multline*}
with $m_i = \begin{cases}
1 & \text{ if } k_i \equiv 1 [2]\\
2 & \text{ if } k_i \equiv 2 [4]\\
4 & \text{ if } k_i\equiv 4 [8]\\
0 & \text{ if } k_i \equiv 0 [8]
\end{cases}$.
\end{theorem}

\begin{proof}
Let us note that the $\chi\left(\left(\Xal{a}{\lambda}\right)^i\right)$'s commute. Hence
\begin{align*}
\hw \ \prod_{i = 1} ^{r} \chi\left(\left(\Xal{a}{\lambda}\right)^{k_i}\right) &= \hw 2^r \dplus \alpha(k) \dplus \alpha(k) \dplus (-1)^{\#\{i, k_i = 1\}} 2^r \dplus \beta(k)\\
&= \frac{1}{8} \left( 2^r \left(1 + (-1)^{\#\{i, k_i = 1\}}\right) + 2 \alpha(k) + 2 \Tr(\beta(k)) \right)
\end{align*}
where $\alpha(k)$ is $0$ if there exists $i$ such that $k_i = 1$ and $2^r$ otherwise, and $\beta(k)$ is the matrix null if there exists $i$ such that $k_i \in \{1,2\}$ and $(-1)^{\#\{i, k_i=4\}}2^r I_2$ otherwise. So, we have
\begin{multline*}
\hw \ \prod_{i = 1} ^{r} \chi\left(\left(\Xal{a}{\lambda}\right)^{k_i}\right)\\
= \begin{cases}
0 &\text{if } 2 \nmid \#\{i, k_i = 1\}\\
2^{r-2} &\text{if } \#\{i, k_i = 1\} \in 2 (\mathbb{N} \setminus\{0\})\\
2^{r-1} &\text{if } \#\{i, k_i = 1\}=0, \#\{i, k_i = 2\} \geq 1\\
2^{r-1} \left(1 + (-1)^{\#\{i, k_i = 4\}}\right) &\text{otherwise} 
\end{cases}
\end{multline*}

Clearly, $\chi\left(\left(\Xal{a}{\lambda}\right)^{0}\right)$ is independent from the other $\chi\left(\left(\Xal{a}{\lambda}\right)^{i}\right)$'s. To study the independence of $\chi\left(\Xal{a}{\lambda}\right)$, let us look at classical cumulants. Let $\{p_1, p_2, \ldots, p_r\}$ be a subset of $\{1,2,4\}$, and $\kappa(p_1, \ldots, p_r)$ be the joint cumulant of $\chi\left(\left(\Xal{a}{\lambda}\right)^{p_1}\right), \ldots, \chi\left(\left(\Xal{a}{\lambda}\right)^{p_r}\right)$. By definition, we have
\begin{align*}
\kappa(p_1, \ldots, p_r) & = \sum_{\pi \in \mathcal{P}_r} (|\pi|-1)!\, (-1)^{|\pi|-1} \prod_{B \in \pi} \hw \ \prod_{k \in B}\chi\left(\left(\Xal{a}{\lambda}\right)^{p_k}\right)
\end{align*}
hence, if $1$ is in $\{p_1, p_2, \ldots, p_r\}$ the cumulant is $0$, it means that $\chi\left(\Xal{a}{\lambda}\right)$ is independent from the others. Moreover,
\[\kappa(2,4) = \hw \chi\left(\left(\Xal{a}{\lambda}\right)^{2}\right) \chi\left(\left(\Xal{a}{\lambda}\right)^{4}\right) - \hw \chi\left(\left(\Xal{a}{\lambda}\right)^{2}\right) \hw \chi\left(\left(\Xal{a}{\lambda}\right)^{4}\right) = 2\]
thus $\chi\left(\left(\Xal{a}{\lambda}\right)^{2}\right)$ and $\chi\left(\left(\Xal{a}{\lambda}\right)^{4}\right)$ are not independent.
\end{proof}

\section{The Sekine Finite Quantum Groups $\KP_n$}

\subsection{Definition}

We will follow the definition of \cite{McCarthy}, but for convenience of the reader, we recall here the notations.

Consider the multi-matrix algebra $\mathcal{A}_n = \bigoplus_{i,j \in \Zn} \CC e_{(i,j)} \oplus \mathcal{M}_n(\CC)$ together with usual multiplication and involution. This is a $2n^2$-dimensional algebra, with basis $\{e_{(i,j)}\}_{i,j \in \Zn} \cup \{E_{i,j}\}_{1 \leq i,j \leq n}$. The unit is naturally
\[\mathds{1} = \sum_{i,j \in \Zn} e_{(i,j)} + \sum_{i = 1} ^{n} E_{i,i}\]

The following defines the coproduct:
\[\cop(e_{(i,j)}) = \sum_{k,l \in \Zn} e_{(k,l)} \otimes e_{(i-k,j-l)} + \frac{1}{n} \sum_{k,l \in \Zn} \eta^{i(k-l)} E_{k,l}\otimes E_{k+j,l+j}\]
\[\cop(E_{i,j}) = \sum_{k,l \in \Zn} e_{(-k,-l)} \otimes \eta^{k(i-j)}E_{i-l,j-l} + \sum_{k,l\in \Zn} \eta^{k(j-i)} E_{i-l,j-l}\otimes e_{(k,l)}\]
with $\eta = e^{\frac{2 \imath \pi}{n}}$ a primitive $n$th root of unity. The counit is given by 
\[\cou \left(\sum_{i,j \in \Zn} x_{(i,j)}e_{(i,j)} + \sum_{1 \leq i,j \leq n} X_{i,j} E_{i,j}\right) = x_{(0,0)}\]
and the antipode satisfies $S(e_{(i,j)}) = e_{(-i,-j)}$ and $S(E_{i,j}) = E_{j,i}$.

This defines a finite quantum group, denoted by $\KP_n = \left( \mathcal{A}_n, \cop\right)$, called Sekine quantum group. We shall also need its Haar state, denoted by $\hwn$ and given by the following formula:
\[\hwn \left(\sum_{i,j \in \Zn} x_{(i,j)}e_{(i,j)} + \hspace{-5pt}\sum_{1 \leq i,j \leq n} X_{i,j} E_{i,j} \right) = \frac{1}{2n^2} \left(\sum_{i,j \in \Zn} x_{(i,j)} + n \sum_{i = 1}^n X_{i,i}\right) \text{ .}\]

\begin{remark}
As noted in \cite{McCarthy}, with this definition, $\KP_2$ is cocommutative and equal to the virtual object $\widehat{\mathbb{D}_4}$, i.e.\ $\mathcal{A}_2 \simeq \CC \mathbb{D}_4$.
\end{remark}

\subsection{Representation theory}

Let us first determine the representation theory of the Sekine finite quantum groups. We list here the irreducible unitary representations of $\KP_n$, for each $n \geq 2$. Note that it depends on the parity of $n$.

\subsubsection{Case $n$ odd}

\begin{theorem}[\cite{McCarthy}]
If $n$ is odd, the finite quantum group $\KP_n$ admits $2n$ one-dimensional non equivalent unitary representations,
\[\forall l \in \{0, 1, \ldots, n-1\}, \ \rho_l^\pm = \sum_{i,j \in \Zn} \eta^{il} e_{(i,j)} \pm \sum_{i = 1}^n E_{i, i+l} \text{ .}\]

It also admits $\frac{n(n-1)}{2}$ non equivalent unitary two-dimensional irreducible representations, indexed by $u \in \{0, 1, \ldots, n-1\}$ and $v \in \{1, 2, \ldots, \frac{n-1}{2}\}$, given by their matrix-coefficients:
\[\Xuv{u}{v}_{11} = \sum_{i,j \in \Zn} \eta^{iu+jv} e_{(i,j)} \hspace{2cm} \Xuv{u}{v}_{12} = \sum_{i = 1}^n \eta^{-iv} E_{i, i+u}\]
\[\Xuv{u}{v}_{21} = \sum_{i = 1}^n \eta^{iv} E_{i, i+u} \hspace{2cm} \Xuv{u}{v}_{22} = \sum_{i,j \in \Zn} \eta^{iu-jv} e_{(i,j)}\]
\end{theorem}

\subsubsection{Case $n$ even}

\begin{theorem}
If $n$ is even, the finite quantum group $\KP_n$ admits $4n$ one-dimensional non equivalent unitary representations,
\[\forall l \in \{0, 1, \ldots, n-1\}, \ \begin{cases} \rho_l^\pm &= \sum\limits_{i,j \in \Zn} \eta^{il} e_{(i,j)} \pm \sum\limits_{i = 1}^n E_{i, i+l}\\\sigma_l^\pm &= \sum\limits_{i,j \in \Zn} (-1)^j \eta^{il} e_{(i,j)} \pm \sum\limits_{i = 1}^n (-1)^i E_{i, i+l} \end{cases} \text{ .}\]

It also admits $\frac{n(n-2)}{2}$ non equivalent unitary two-dimensional irreducible representations, indexed by $u \in \{0, 1, \ldots, n-1\}$ and $v \in \{1, 2, \ldots, \frac{n}{2}-1\}$, given by their matrix-coefficients:
\[\Xuv{u}{v}_{11} = \sum_{i,j \in \Zn} \eta^{iu+jv} e_{(i,j)} \hspace{2cm} \Xuv{u}{v}_{12} = \sum_{i = 1}^n \eta^{-iv} E_{i, i+u}\]
\[\Xuv{u}{v}_{21} = \sum_{i = 1}^n \eta^{iv} E_{i, i+u} \hspace{2cm} \Xuv{u}{v}_{22} = \sum_{i,j \in \Zn} \eta^{iu-jv} e_{(i,j)}\]
\end{theorem}

\begin{proof}
The computations done for $\rho_l^\pm$  and $\Xuv{u}{v}$ in the odd case in \cite{McCarthy} are still valid.

It is clear that $\cou(\sigma_l^\pm) = 1$. By \cite[Proposition 3.1.7]{Timmermann}, it remains to prove that $\sigma_l^\pm$ are group-like elements. The same steps as for $\rho_l^\pm$ give
\begin{align*}
\cop(\sigma_l^\pm) = \sum_{i,j,s,t \in \Zn}\!\!\!(-1)^j &\eta^{il} e_{(s,t)} \otimes e_{(i-s,j-t)}\\
&+ \sum_{\substack{m = 1\\j \in \Zn}}^n(-1)^j E_{m,m+l} \otimes E_{m+j, m+j+l}\\
&\pm \sum_{\substack{m = 1\\s,t \in \Zn}}^n(-1)^m \eta^{sl} e_{(s,t)} \otimes E_{m+t, m+t+l}\\
&\pm \sum_{\substack{m = 1\\s,t \in \Zn}}^n(-1)^m \eta^{sl} E_{m-t,m-t+l} \otimes e_{(s,t)} \text{ .}
\end{align*}
On the other hand, we have
\begin{align*}
\sigma_l^\pm \otimes \sigma_l^\pm = \sum_{i,j,s,t \in \Zn}\!\!\!(-1)^{j+t} &\eta^{(i+s)l} e_{(i,j)} \otimes e_{(s,t)}\\
&+ \sum_{s,t = 1}^n (-1)^{s+t} E_{s,s+l} \otimes E_{t,t+l}\\
&\pm \sum_{\substack{m = 1\\s,t \in \Zn}}^n(-1)^{t+m} \eta^{sl} e_{(s,t)} \otimes E_{m,m+l}\\
&\pm \sum_{\substack{m = 1\\s,t \in \Zn}}^n(-1)^{m+j} \eta^{sl} E_{m,m+l} \otimes e_{(s,t)}
\end{align*}
which is the same, up to re-indexation ($s+i \to i$, $t+j \to j$, $i \to s$ and $j \to t$ in the first term, $t-s \to j$ in the second one, $m-t \to m$ in the third one and $m+t \to m$ in the last one).
\end{proof}

\begin{remark}
This is another way to see that the Kac-Paljutkin $\KP$ finite quantum group is different from $\KP_2$, since they do not have the same representation theory. The quantum group $\KP$ admits four one-dimensional representations and one two-dimensional irreducible representation, whereas the Sekine quantum group $\KP_2$ admits eight one-dimensional representations and no two-dimensional irreducible representation.
\end{remark}

\subsection{Characters}

We will now consider the traces of powers of the two-dimensional irreducible representations $\Xuv{u}{v}$. In the following, we (incorrectly) call $k$th character associated to $\Xuv{u}{v}$ the trace of the $k$th power of this representation.

\begin{lemma}
\label{lemDiag}
The even powers of $\Xuv{u}{v}$ are diagonal matrices.
\end{lemma}

\begin{proof}
By definition and orthogonality of the $e_{(i,j)}$'s and the $E_{i,j}$'s, we have
\[\left(\left(\Xuv{u}{v}\right)^2 \right)_{12} = \Xuv{u}{v}_{11} \Xuv{u}{v}_{12} + \Xuv{u}{v}_{12}\Xuv{u}{v}_{22} = 0_{\KP_n}\]
\[\left(\left(\Xuv{u}{v}\right)^2 \right)_{21} = \Xuv{u}{v}_{21} \Xuv{u}{v}_{11} + \Xuv{u}{v}_{22}\Xuv{u}{v}_{21} = 0_{\KP_n}\]
so $\left(\Xuv{u}{v}\right)^2$ is a diagonal matrix and therefore $\left(\Xuv{u}{v}\right)^{2k} = \left(\left(\Xuv{u}{v}\right)^2\right)^k$ is also a diagonal matrix in $\mathcal{M}_2(\mathcal{A}_n)$.
\end{proof}

\begin{proposition}
\label{PropChar}
The characters associated to $\Xuv{u}{v}$ are
\begin{multline*}
\chi\left(\left(\Xuv{u}{v}\right)^k\right) = 2 \sum_{s,t \in \Zn} \eta^{ksu} \cos\left(\frac{2 k t v \pi}{n}\right) e_{(s,t)}\\
+ \  \indic_{2\mathbb{Z}}(k)\;2 \cos\left(\frac{k v^2 \pi }{n}\right)\sum_{r = 1} ^{n} E_{r, r+ku}
\end{multline*}
where $\indic_{2\mathbb{Z}}$ denotes the indicator function of the even integers.
\end{proposition}

\begin{proof}
By definition, we have
\begin{align*}
\left(\left(\Xuv{u}{v}\right)^2 \right)_{11} &= (\Xuv{u}{v}_{11})^2 + \Xuv{u}{v}_{12}\Xuv{u}{v}_{21}\\
&= \sum_{i,j,k,l \in \Zn} \eta^{iu+jv} e_{(i,j)} \eta^{ku+lv} e_{(k,l)} + \sum_{i,j = 1}^n \eta^{-iv} E_{i, i+u} \eta^{jv} E_{j, j+u}\\
&= \sum_{s,t \in \Zn} \eta^{2su+2tv} e_{(s,t)} + \sum_{r = 1}^n \eta^{v^2} E_{r, r+2u}
\end{align*}
and by similar calculations, and by Lemma \ref{lemDiag} we have
\[\left(\left(\Xuv{u}{v}\right)^{2k} \right)_{11} = \sum_{s,t \in \Zn} \eta^{2ksu+2ktv} e_{(s,t)} + \sum_{r = 1}^n \eta^{kv^2} E_{r, r+2ku}\]
\[\left(\left(\Xuv{u}{v}\right)^{2k} \right)_{22} = \sum_{s,t \in \Zn} \eta^{2ksu-2ktv} e_{(s,t)} + \sum_{r = 1}^n \eta^{-kv^2} E_{r, r+2ku}\]
which leads to the result for even powers.

Let $k = 2p + 1$, then $\left(\Xuv{u}{v}\right)^{2p + 1} =\left(\Xuv{u}{v}\right)^{2p} \Xuv{u}{v}$, hence we obtain
\begin{align*}
\left(\left(\Xuv{u}{v}\right)^{k} \right)_{11} &= \left(\left(\Xuv{u}{v}\right)^{2p} \right)_{11} \left(\Xuv{u}{v}\right)_{11} + \left(\left(\Xuv{u}{v}\right)^{2p} \right)_{12} \left(\Xuv{u}{v}\right)_{21}\\
&=\left(\sum_{i,j \in \Zn} \eta^{2piu+2pjv} e_{(i,j)} + \sum_{i = 1}^n \eta^{pv^2} E_{i, i+2pu}\right)\hspace{-5pt}\sum_{k,l \in \Zn}\hspace{-5pt}\eta^{ku+lv} e_{(k,l)}\\
&= \sum_{s,t \in \Zn} \eta^{ksu+ktv} e_{(s,t)}
\end{align*}
and similarly $\left(\left(\Xuv{u}{v}\right)^{k} \right)_{22} = \sum\limits_{s,t \in \Zn} \eta^{ksu-ktv} e_{(s,t)}$ which leads to the result for odd powers.
\end{proof}

\subsection{Character spaces}

Let us look more precisely at relations between the characters before to study their asymptotic distributions.

\subsubsection{Algebra of characters}

\begin{proposition}
\label{PropSpace}
The algebra of characters, generated by all one-dimensional representations and all the $\chi(\Xuv{u}{v})$, contains all $\chi\left(\left(\Xuv{u}{v}\right)^{k}\right)$, $k \geq 1$.
\end{proposition}

\begin{proof}
If $n$ does not divide $kv$, consider $w$ the absolute value of $kv\pmod{n}$. If $w$ belongs to $\{1, \ldots, \lfloor\frac{n-1}{2}\rfloor\}$, then
\[\chi\left(\left(\Xuv{u}{v}\right)^{k}\right) = \chi\left(\Xuv{ku}{w}\right) + \indic_{2\mathbb{Z}}(k) \cos\left(\frac{k v^2 \pi}{n}\right) (\rho_{ku}^+ - \rho_{ku}^-)\]
otherwise, $n$ is even and $w = \frac{n}{2}$, and
\[\chi\left(\left(\Xuv{u}{v}\right)^{k}\right) = \sigma_{ku}^+ + \sigma_{ku}^- + \indic_{2\mathbb{Z}}(k) \indic_{2\mathbb{Z}}(v) (-1)^{\frac{v}{2}}(\rho_{ku}^+ - \rho_{ku}^-) \text{ .}\]

If $kv = an$ then, for $s$ the sign of $(-1)^{av}$,
\[\chi\left(\left(\Xuv{u}{v}\right)^{k}\right) = \begin{cases}\rho_{ku}^+ + \rho_{ku}^- &\text{ if } k \text{ is odd}\\
2 \rho_{ku}^s &\text{ if } k \text{ is even}\end{cases} \text{ .}\]
\end{proof}

\begin{remark}
This proposition gives us another reason to call character the $\chi\left(\left(\Xuv{u}{v}\right)^k\right)$ since, for the classical groups, every linear combination $\chi$, with coefficients in $\mathbb{Z}$, of characters such that $\chi(e) > 0$ is again a character. Let us note once again that this is not true in general for quantum groups, a counterexample is given by the dual quantum group $\widehat{\KP_n}$ in subsection \ref{ssDual}.
\end{remark}

\begin{proposition}
\label{PropComm}
The algebra of characters is commutative if $n$ is odd.
\end{proposition}

\begin{proof}
We have for $s \neq s^\prime$ or $t \neq t^\prime$, $e_{(s,t)} e_{(s^\prime, t^\prime)} = 0_{\KP_n}$, $e_{(s,t)}^2 = e_{(s,t)}$ and $e_{(s,t)} E_{i,j} = E_{i,j} e_{(s,t)} = 0_{\KP_n}$. On the other hand, we also have, for all natural numbers $a$, $b$,
\[\left(\sum_{m = 1}^{n} E_{m, m+a}\right)\hspace{-5pt}\left(\sum_{\mu = 1}^{n} E_{\mu, \mu+b}\right)\hspace{-5pt}= \sum_{r = 1}^n E_{r, r+a+b} =\hspace{-5pt}\left(\sum_{\mu = 1}^{n} E_{\mu, \mu+b}\right)\hspace{-5pt}\left(\sum_{m = 1}^{n} E_{m, m+a}\right)\]
which leads to the commutativity of the algebra.
\end{proof}

\begin{remark}
If $n$ is even, the subalgebra generated by the $\rho_l^\pm$'s and the $\chi(\Xuv{u}{v})$ is also commutative. But the $\sigma_l^\pm$'s do not commute with the $\rho_l^\pm$'s, since
\begin{multline*}
\left(\sum_{m = 1}^{n} (-1)^m E_{m, m+a}\right)\hspace{-5pt}\left(\sum_{\mu = 1}^{n} E_{\mu, \mu+b}\right)\hspace{-5pt}= \sum_{r = 1}^n (-1)^r E_{r, r+a+b}\\
\neq \left(\sum_{\mu = 1}^{n} E_{\mu, \mu+b}\right)\hspace{-5pt}\left(\sum_{m = 1}^{n} (-1)^m E_{m, m+a}\right) = \sum_{r = 1}^n (-1)^{r+b} E_{r, r+b+a}
\end{multline*}
when $b$ is odd.
\end{remark}

\subsubsection{A commutative (sub)algebra}

For all $n$, the algebra generated by all the $\chi(\Xuv{u}{v})$'s and all the $\rho_l^\pm$'s is a classical commutative algebra. Thus, by the spectral theorem, it is $*$-isomorphic to $\CC^k$ for some $k \in \mathbb{N}$.  In order to keep track of the restriction of the Haar state, we can view it as a subalgebra of some $L^\infty(\Omega_n, \mu_n)$, for a compact space $\Omega_n$ and a probability distribution $\mu_n$. It means that we can see these characters as classical random variables on the classical probability space $(\Omega_n, \mu_n)$.

By the Gelfand-Naimark Theorem, $\Omega_n$ is the spectrum of the algebra, it means the set of all its characters, which are all the non zero $*$-multiplicative linear forms. To determine this space and the measure $\mu_n$, we need to investigate deeper the structure of the $*$-algebra
\[\mathcal{C}_n = \alg\left\{\rho_l^\pm, \chi\left(\Xuv{u}{v}\right), 0 \leq l, u \leq n-1, 1 \leq v \leq \left\lfloor \frac{n-1}{2}\right\rfloor\right\} \text{ .}\]

\begin{lemma}
\label{LemSpan}
For all non negative integers $k$ and $l$ smaller than $n-1$, $\rho_k^+ \rho_l^+ = \rho_{k+l}^+$, $\rho_k^- \rho_l^- = \rho_{k+l}^+$ and $\rho_k^+ \rho_l^- = \rho_{k+l}^- = \rho_k^- \rho_l^+$, where the sum $k+l$ is taken in $\Zn$. We also have that, for all integer $v$ between $1$ and $\lfloor \frac{n -1}{2} \rfloor$, and for all integer $u$ between $0$ and $n-1$, there exist  integers $a_i \in \mathbb{Z}$ such that
\[\chi\left(\Xuv{u}{v}\right) = \rho_u^\pm \left(\left(\chi\left(\Xuv{0}{1}\right)\right)^v + a_{v-2} \left(\chi\left(\Xuv{0}{1}\right)\right)^{v-2} + \ldots + a_1 \chi\left(\Xuv{0}{1}\right)\right)\]
if $v$ is odd, or, if $v$ is even
\[\chi\left(\Xuv{u}{v}\right) = \rho_u^\pm\!\left(\left(\chi\left(\Xuv{0}{1}\right)\right)^v+ a_{v-2} \left(\chi\left(\Xuv{0}{1}\right)\right)^{v-2}\!\!+ \ldots + a_0 \left(\rho_0^+ + \rho_0^-\right)\right) \text{ .}\]
\end{lemma}

\begin{proof}
The first assertion follows from the definition of the componentwise multiplication.

For the second part of the lemma, by the same way, we easily see that we have $\rho_u^\pm \chi\left(\Xuv{0}{v}\right) = \chi\left(\Xuv{u}{v}\right)$. Therefore, we only need to consider the $\chi\left(\Xuv{0}{v}\right)$.

Let us note that the Tchebychev polynomials $T_n$ satisfy, for all real $\theta$ and all natural integer $n$, $T_n(\cos(\theta)) = \cos(n\theta)$ , so we have
\begin{align*}
\chi\left(\Xuv{0}{v}\right) &= 2 \sum_{s,t \in \Zn} \cos\left(\frac{2\pi t v}{n}\right) e_{(s,t)}\\
&= 2 \sum_{s,t \in \Zn} T_v\!\left(\cos\left(\frac{2\pi t}{n}\right)\right) e_{(s,t)}\\
&= 2 \  \tilde{T}_v\!\left( \frac{\chi\left(\Xuv{0}{1}\right)}{2}\right)
\end{align*}
thanks to the componentwise multiplication, where the constant term in $\tilde{T}_v\!\left(\frac{\chi\left(\Xuv{0}{1}\right)}{2}\right)$ is $0$ or $(\rho_0^+ + \rho_0^-) = \sum\limits_{s,t \in \Zn} e_{(s,t)}$ but not $\indic_{\KP_n}$, what we should have if we substitute $\chi\left(\Xuv{0}{1}\right)$ in $T_v(X)$.

Moreover, the Tchebychev polynomial $T_v$ has degree $v$ with leading coefficient $2^{v-1}$, and all its coefficients are integers. It is symmetric if $v$ is even, or antisymmetric if $v$ is odd.
\end{proof}

This Lemma means that
\[\mathcal{C}_n = \alg \left\{\rho_1^+, \rho_1^-, \chi\left(\Xuv{0}{1}\right)\right\} \text{ .}\]

\begin{remark}
Since, when $n$ is odd, we have $(\rho_1^-)^{n+1} = \rho_1^+$, the corresponding $\mathcal{C}_n$ is generated (as an algebra) by $\left\{\rho_1^-, \chi\left(\Xuv{0}{1}\right)\right\}$.
\end{remark}

Thus, by the properties of the characters, they are defined by their values on $\rho_1^+$, $\rho_1^-$ and  $\chi\left(\Xuv{0}{1}\right)$. Moreover, $\sigma(a) = \left\{\omega(a), \omega \in \Omega\right\}$, so $\Omega_n$ is fixed by the spectra of the three elements and some relations. Let us note that $\mathcal{C}_n$ is a subalgebra of $\KP_n$ containing the unit. Hence, the spectrum with respect to $\mathcal{C}_n$ is the spectrum with respect to $\KP_n$. Direct calculations show that:
\begin{align*}
\sigma\left(\chi\left(\Xuv{0}{1}\right)\right) &= \left\{2 \cos\left(\frac{2t\pi}{n}\right), \  t \in \Zn \right\} \cup \{0\} \text{ ,}\\
\sigma\left(\rho_1^+\right) &= \left\{\eta^s,\  s \in \Zn \right\} \text{ ,}\\
\sigma\left(\rho_1^-\right) &= \left\{\eta^s,\  s \in \Zn \right\} \cup \left\{-\eta^s,\  s \in \Zn \right\} \text{ .}
\end{align*}

To determine $\Omega_n$, let us note that
\[(\rho_1^+)^2 = (\rho_1^-)^2 = \rho_2 ^+ \text{ and }\rho_1^+ \chi\left(\Xuv{0}{1}\right) = \rho_1^- \chi\left(\Xuv{0}{1}\right) = \chi\left(\Xuv{1}{1}\right)\]
which leads to the relations
\[\begin{cases}
\forall \omega \in \Omega_n, \ \omega\left(\rho_1^+\right) = \pm \omega\left(\rho_1^-\right)\\
\forall \omega \in \Omega_n, \ \omega\left(\rho_1^+\right) = - \omega\left(\rho_1^-\right) \Rightarrow \omega\left(\chi\left(\Xuv{1}{1}\right)\right) = 0
\end{cases} \text{ .}\]

Finally, we get the following result
\begin{theorem}
\label{ThOmega}
For all $n$, $\mathcal{C}_n$, equipped with the Haar state, can be viewed as an algebra of random variables on the probability space
\begin{multline*}
\Omega_n = \left\{\eta^s,\  s \in \Zn \right\} \times \{1\} \times \left(\left\{2 \cos\left(\frac{2t\pi}{n}\right), \  t \in \Zn \right\} \cup \{0\}\right)\\
\sqcup \left\{-\eta^s,\  s \in \Zn \right\} \times \{-1\} \times \{0\}
\end{multline*}
endowed with the measure
\[\mu_n = \left(\frac{\indic_{4\mathbb{Z}}(n)}{2n(p+1)} + \frac{1-\indic_{4\mathbb{Z}}(n)}{2n(p+2)}\right) \sum_{\substack{\omega = (a,b,c) \in \Omega_n\\b = 1}}\!\!\!\!\delta_\omega + \frac{1}{2n} \sum_{\substack{\omega = (a,b,c) \in \Omega_n\\b = -1}}\!\!\!\!\delta_\omega\]
where $p = \left\lfloor \frac{n}{2}\right\rfloor$ and $\indic_{4\mathbb{Z}}$ is the indicator function of $4\mathbb{Z}$, thanks to the Gelfand transform,
\begin{align*}
\mathcal{F}\colon &\mathcal{C}_n \to L^\infty(\Omega_n, \mu_n)\\
&x \mapsto \left(\hat{x} \colon \omega \mapsto \omega(x)\right)
\end{align*}
given by the following formula
\[\widehat{\rho_1^-} (a,b,c) = a \ , \ \widehat{\rho_1^+} (a,b,c) = ab \ , \ \widehat{\chi\left(\Xuv{0}{1}\right)} (a,b,c) = c \text{ .}\]
\end{theorem}

\subsection{Asymptotic laws}

\begin{definition}
We say that a complex random variable $Z$ is $\CC$-arcsine$(\alpha)$ distributed, if it admits $z \mapsto \indic_{\alpha\uD}(z) \frac{1}{\pi^2\sqrt{\alpha^2-|z|^2}}$ as density function. Let us denote by $\mu_{\CC-arc(\alpha)}$ the corresponding distribution.
\end{definition}

\begin{lemma}
\label{lemMomCarc}
If $Z$ is a $\CC$-arcsine$(2)$ random variable, for all $k$ and $l$, we have
\[\esp[Z^k \bar{Z}^l] = \begin{cases} \binom{2k}{k} &\text{ if } k = l\\ 0 & \text{ otherwise}\end{cases} \text{ .}\]
\end{lemma}

\begin{proof}
We need to compute
\[\frac{1}{\pi^2} \int_{0}^2 \int_0^{2\pi} \frac{r^{k+l} e^{\imath (k-l)\theta}}{\sqrt{4 - r^2}}\;\mathrm{d}\theta\;\mathrm{d}r \text{ .}\]
The integral with respect to $\theta$ vanishes, except if $k = l$. In this case, we are left with the $2k$th moment of the arcsine distribution on $[-2\;;\;2]$.
\end{proof}

This helps us to find the $*$-distribution of characters associated to the irreducible representations $\Xuv{u}{v}$.

\begin{theorem}
\label{ThOdd}
For all $u, v \geq 1$, $\chi\left(\Xuv{u}{v}\right)$ is asymptotically (when $n \to +\infty$) a $\left(\frac{1}{2} \delta_0 + \frac{1}{2}\mu_{\CC-arc(2)}\right)$-distributed random variable, and $\chi\left(\Xuv{0}{v}\right)$ admits asymptotically (when $n \to +\infty$) the $*$-distribution $\frac{1}{2} \delta_0 + \frac{1}{2}\mu_{arc(-2,2)}$, where $\mu_{arc(-2,2)}$ represents the classical arcsine distribution on the open interval $(-2\;,\;2)$.

Moreover, the same holds for all $\chi\left(\left(\Xuv{u}{v}\right)^{k}\right)$, with $k$ odd.
\end{theorem}

\begin{proof}
We only do the computation for $\chi\left(\Xuv{u}{v}\right)$. The case $\chi\left(\left(\Xuv{u}{v}\right)^{k}\right)$, with $k>1$ odd, is similar.

Let us remember that, by Proposition \ref{PropChar},
\[\chi\left(\Xuv{u}{v}\right) = 2 \sum_{s,t \in \Zn} \eta^{su} \cos\left(\frac{2 t v \pi}{n}\right) e_{(s,t)} \text{ .}\]
Therefore, $\chi\left(\Xuv{u}{v}\right)$ commutes with $\chi\left(\Xuv{u}{v}\right)^*$, and for all $m \geq 1$
\[\left(\chi\left(\Xuv{u}{v}\right)\right)^m = 2^m \sum_{s,t \in \Zn} \eta^{msu} \cos\left(\frac{2 t v \pi}{n}\right)^m e_{(s,t)}\]
so we have
\begin{align*}
\hwn \left(\chi\left(\Xuv{u}{v}\right)\right)&^{r_ 1}\left(\chi\left(\Xuv{u}{v}\right)\right)^{*r_*}\\
&= \frac{2^{r_1+r_*}}{2 n^2} \sum_{s,t \in \Zn} \eta^{(r_1-r_*)su} \cos\left(\frac{2 t v \pi}{n}\right)^{r_1+r_*}\\
&= \frac{1}{2n^2} \left(\sum_{s \in \Zn} \eta^{(r_1-r_*)su}\right)\hspace{-2pt}\left(\sum_{l = 0}^{r_1+r_*} \binom{r_1+r_*}{l}\hspace{-5pt}\sum_{t \in \Zn} \eta^{(2l- (r_1+r_*))tv} \right)\\
&= \frac{1}{2} \indic_{n \mathbb{Z}}((r_1-r_*)u)\sum_{l=0}^{r_1+r_*} \binom{r_1+r_*}{l} \indic_{n \mathbb{Z}}((2l-(r_1+r_*))v) \text{ .}
\end{align*}

For all $u \geq 1$, $r_1 \neq r_*$, $n \geq u(r_1-r_*) + 1$, we have $0 < u(r_1-r_*) < n$, so
\[\hwn \left(\chi\left(\Xuv{u}{v}\right)\right)^{r_ 1}\left(\chi\left(\Xuv{u}{v}\right)\right)^{*r_*} = 0 \text{ .}\]
Otherwise, if $r_1 = r_*= r \geq 1$, for $n$ great enough ($n \geq 2rv + 1$), we obtain
\[\hwn \left(\chi\left(\Xuv{u}{v}\right)\right)^{r}\left(\chi\left(\Xuv{u}{v}\right)\right)^{*r} = \frac{1}{2} \binom{2r}{r} \text{ .}\]
Hence, if $u \neq 0$, $\chi\left(\Xuv{u}{v}\right)$ is asymptotically a $\left(\frac{1}{2} \delta_0 + \frac{1}{2}\mu_{\CC-arc(2)}\right)$-distributed random variable.

If $u = 0$, let us note that $\chi\left(\Xuv{0}{v}\right)$ is selfadjoint, and
\[\hwn \left(\chi\left(\Xuv{u}{v}\right)\right)^{m} = \frac{1}{2} \sum_{l=0}^{m} \binom{m}{l} \indic_{n \mathbb{Z}}((2l-m)v) \text{ .}\]
Hence, for $n$ great enough, if $m$ is odd, the moment vanishes, and if $m = 2p$, it is $\frac{1}{2}\binom{m}{p}$, which corresponds to the distribution $\frac{1}{2} \delta_0 + \frac{1}{2}\mu_{arc(-2,2)}$.
\end{proof}

\begin{theorem}
\label{ThEven}
For all integers $u, v \geq 1$ and any even $k$, $\chi\left(\left(\Xuv{u}{v}\right)^{k}\right)$ is asymptotically (when $n \to +\infty$) a $\left(\frac{1}{2} \mathcal{U}(2\uT) + \frac{1}{2}\mu_{\CC-arc(2)}\right)$-distributed random variable, and $\chi\left(\left(\Xuv{0}{v}\right)^{k}\right)$ admits $\frac{1}{2} \delta_2 + \frac{1}{2}\mu_{arc(-2,2)}$ as asymptotic $*$-distribution (when $n \to +\infty$).
\end{theorem}

\begin{proof}
The only difference with the characters in the preceding theorem is the part "$+ 2 \cos\left(\frac{k v^2 \pi }{n}\right)\sum\limits_{r = 1} ^{n} E_{r, r+ku}$" in the even characters.

By the properties of the multiplication in $\KP_n$ we only need to study it, and show that its moments are asymptotically one half of those for a random variable uniformly distributed on $2 \uT$.

Let us denote this normal matrix by $M_{k, u, v, n}$. Then, for all $m \geq 1$
\[\begin{cases}
M_{k, u, v, n}^m &= 2^m \cos\left(\frac{k v^2 \pi }{n}\right)^m\sum\limits_{r = 1} ^{n} E_{r, r+mku}\\
(M_{k, u, v, n}^*)^m &= 2^m \cos\left(\frac{k v^2 \pi }{n}\right)^m\sum\limits_{r = 1} ^{n} E_{r, r-mku}
\end{cases} \text{ .}\]
Hence, we have
\begin{align*}
\hwn M_{k, u, v, n}^\alpha (M_{k, u, v, n}^*)^\beta &= \frac{1}{2n^2} n 2^{\alpha + \beta} \cos\left(\frac{k v^2 \pi }{n}\right)^{\alpha + \beta} \Tr\left(\sum_{r = 1}^n E_{r, r+ku(\alpha-\beta)} \right)\\
&= \frac{2^{\alpha + \beta}}{2n} \cos\left(\frac{k v^2 \pi }{n}\right)^{\alpha + \beta} \sum_{r = 1}^n \indic_{n\mathbb{Z}}(ku(\alpha-\beta))\\
&= \frac{1}{2} \left(2\cos\left(\frac{k v^2 \pi }{n}\right)\right)^{\alpha + \beta} \indic_{n\mathbb{Z}}(ku(\alpha-\beta)) \text{ .}
\end{align*}

For $u = 0$, we get $\frac{1}{2} \left(2\cos\left(\frac{k v^2 \pi }{n}\right)\right)^{\alpha + \beta}$ which goes to one half of $2^{\alpha + \beta}$.

For $u \geq 1$ and $n$ great enough, this moment goes to one half of the moment of $\mathcal{U}(2 \uT)$, which is zero if $\alpha$ is different from $\beta$, $2^{2\alpha}$ otherwise. This completes the proof of the theorem.
\end{proof}

\begin{remark}
Let us note that $\sum\limits_{r = 1} ^{n} E_{r, r+ku}$ is a permutation matrix, whose eigenvalues are the $n$th roots of unity. Therefore the eigenvalues of the considered matrix $2 \cos\left(\frac{k v^2 \pi }{n}\right)\sum\limits_{r = 1} ^{n} E_{r, r+ku}$ are the $n$th roots of unity multiplied by $2 \cos\left(\frac{k v^2 \pi }{n}\right)$. So when $n$ goes to infinity, the eigenvalues come out uniformly on $2 \uT$ if $u$ is not zero, which corresponds to the result of the  Theorem \ref{ThEven}.
\end{remark}

\begin{remark}
If we let $n$ go to $\infty$ in Theorem \ref{ThOmega}, we see that, $(\Omega_n, \mu_n)$ converges to $(\Omega, \mu)$ where
\[\Omega = \uT \times \{1\} \times [-2\;;\;2] \sqcup \uT \times \{-1\} \times \{0\} \text{ , } \mu = \frac{1}{2} \mu_{\CC-arc(2)} + \frac{1}{2} \mathcal{U}(\uT) \text{ .}\]
And, we can check easily that
\[\widehat{\rho_1^-}, \widehat{\rho_1^+} \sim \mathcal{U}(\uT) \text{ and } \widehat{\chi\left(\Xuv{0}{1}\right)} \sim \frac{1}{2} \mu_{arc(-2,2)} + \frac{1}{2} \delta_0 \text{ .}\]
\end{remark}

\begin{remark}
The same type of computations for the dihedral group $D_{2n}$ gives similar results. The dihedral group $D_{2n}$ admits indeed $\left\lfloor \frac{n-1}{2} \right\rfloor$ two-dimensional non equivalent unitary irreducible representations $\sigma_k$ given by
\[s \mapsto \begin{pmatrix}0&1\\1&0\end{pmatrix} \text{ , } t \mapsto \begin{pmatrix} \eta^k&0\\0&\eta^{-k}\end{pmatrix}\]
where $s$ and $t$ generate $D_{2n}$ with $s$ of order $2$ and $t$ of order $n$. Let us denote by $\chi_{k,l}$ the class function $\Tr\left( \sigma_k(\cdot)^l\right)$. Then by the moments method, the asymptotic law of $\chi_{k,l}$ is $\frac{1}{2} \left(\mu_{arc(-2,2)} + \delta_2\right)$ if $l$ is even and positive, $\frac{1}{2} \left(\mu_{arc(-2,2)} + \delta_0\right)$ if $l$ is odd or $\delta_2$ when $l = 0$.
\end{remark}

\subsection{Asymptotic pairwise independence}

Since we only consider $\chi\hspace{-2pt}\left(\Xuv{u}{v}\right)$ or its adjoint, we work in a commutative setting, and we can use classical cumulants $\kappa$. Let $b$, $d$, $k$, $l$, be natural integers, $a$ and $c$ be non negative integers and $e$, $f$ be in $\{1, *\}$. Then direct calculation leads to
\[\lim_{n \to +\infty} \kappa\left(\chi\left(\left(\Xuv{a}{b}\right)^k\right)^e, \chi\left(\left(\Xuv{c}{d}\right)^l\right)^f\right) = \delta_{ka, lc} \delta_{\{e,f\}, \{1, *\}}\left(\delta_{kb, ld} + 2 \indic_{(2\mathbb{Z})^2}(k,l) \right)\]
which proves
\begin{proposition}
\label{PropIndPair}
Let $a$, $b$, $c$, $d$, $e$, $f$, $k$ and $l$ be as above. The followings are equivalent
\begin{enumerate}
\item $\chi\left(\left(\Xuv{a}{b}\right)^k\right)^e$ and $\chi\left(\left(\Xuv{c}{d}\right)^l\right)^f$ are asymptotically independent
\item at least one of the following conditions holds
\begin{itemize}
\item $ka \neq lc$
\item $ e = f$
\item $k$ or $l$ is odd, and $kb \neq ld$
\end{itemize}
\end{enumerate}
\end{proposition}

\subsection{Dual groups}
\label{ssDual}
We can define a notion of duality for a finite quantum group $\mathbb{G} = (C(\mathbb{G}), \Delta)$. This dual, denoted $\hat{\mathbb{G}} = (C(\hat{\mathbb{G}}), \hat{\cop})$, is again a finite quantum group. The algebra  $C(\hat{\mathbb{G}})$ is the set of all linear forms defined on $C(\mathbb{G})$. It is also isomorphic to the direct sum of the non-equivalent unitary irreducible corepresentations of $C(\mathbb{G})$, thanks to the Fourier transform. All the structures are defined by duality.

In particular, for the case of the Sekine family, let us denote by $e^{(i,j)}$ and $E^{i,j}$ the elements of the dual basis. Then, we have the coproduct, dual of the multiplication on $\mathcal{A}_n$, given by
\[\hat{\cop}(e^{(i,j)}) = e^{(i,j)} \otimes e^{(i,j)} \ \text{ and } \ \hat{\cop}(E^{i,j})= \sum_{k = 1}^n E^{i,k} \otimes E^{k,j}\]
and the Haar state, dual of the counit on $\mathcal{A}_n$,
\[\int_{\widehat{\KP_n}}\!\!\left(\sum_{i,j \in \Zn} x_{(i,j)}e^{(i,j)} + \sum_{1 \leq i,j \leq n} X_{i,j} E^{i,j}\right) = x_{(0,0)} \text{ .}\]

We clearly obtain that the unitary irreducible representations of $\widehat{\KP_n}$ are the $n^2$ $e^{(i,j)}$'s and the $n$-dimensional representation $\hat{\mathfrak{X}}$, given by its matrix-elements $\hat{\mathfrak{X}}_{i,j} = E^{i,j}$. Let us note that  $\hat{\mathfrak{X}}$ is a fundamental representation of $\widehat{\KP_n}$. Moreover, these representations are non-equivalent, so there is no other non-equivalent unitary irreducible representation.

The characters, $e^{(i,j)}$, $i, j \in \Zn$, and $\chi(\hat{\mathfrak{X}}) = \sum\limits_{i = 1}^{n} E^{i,i}$, generate a commutative algebra, for the product in $C(\widehat{\KP_n})$. This central algebra is the linear span of these characters and may not contain all the traces of powers of $\hat{\mathfrak{X}}$, given for all positive integer $k$ by
\[\chi\left(\hat{\mathfrak{X}}^{k}\right) =
\begin{cases}
\sum\limits_{\substack{s,t \in \Zn\\kt \equiv 0 \!\!\!\mod n}} E^{s,s+t} &\text{ if } k \text { is odd}\\
\sum\limits_{\substack{s,t \in \Zn\\kt \equiv 0 \!\!\!\mod n}} \eta^{-st} e^{(s, t\frac{k}{2})} &\text{ if } k \text { is even}
\end{cases} \text{ .}\]
Note that $\chi\left(\hat{\mathfrak{X}}^{k}\right)$ equals $\chi(\hat{\mathfrak{X}})$ when $k$ is odd and $gcd(k,n) = 1$. If $gcd(k,n)$ is greater than $1$ for some odd $k$, then $\chi\left(\hat{\mathfrak{X}}^{k}\right)$ is not in the character space.

However, these traces are all self-adjoint, and commute. Using the corepresentations basis, we can get that, for all positive integers $k_1, k_2, \ldots, k_r$
\[\int_{\widehat{\KP_n}} \!\!\chi\!\left(\hat{\mathfrak{X}}^{k_1}\right)\!\ldots \chi\!\left(\hat{\mathfrak{X}}^{k_r}\right) = n^{r-2} \frac{1 +\!(-1)^{\sum\limits_{i = 1}^{r} k_i}}{2}\ \#\!\left\{t \in \Zn, \forall 1\leq i \leq r, n\!\mid\!k_i t\right\} \text{ .}\]

In particular, by the moments method, the normalized character $\frac{1}{n}\chi(\hat{\mathfrak{X}})$ admits the $*$-distribution $\frac{1}{2n^2} \left(\delta_{-1} + \delta_{1}\right) + (1 - \frac{1}{n^2}) \delta_0$.
Moreover, let us note that $\#\left\{t \in \Zn, \forall 1\leq i \leq r, n \mid k_i t\right\} = \gcd(n, k_1, \ldots, k_r)$. Thus, for any $k \geq 1$, the normalized trace $\frac{1}{n}\chi(\hat{\mathfrak{X}}^{k})$ admits the $*$-distribution $\frac{\gcd(k,n)}{2n^2} \left(\delta_{-1} + \delta_{1}\right) + (1 - \frac{\gcd(k,n)}{n^2}) \delta_0$ if $k$ is odd, and $\frac{\gcd(k,n)}{n^2} \delta_{1} + (1 - \frac{\gcd(k,n)}{n^2}) \delta_0$ if $k$ is even.

\section*{Acknoledgment.} The author is grateful to Professor Uwe Franz for useful conversations and to the anonymous referee whose comments help to improve this article. This work was supported by the French "Investissements d'Avenir" program, project ISITE-BFC (contract ANR-15-IDEX-03)

\end{document}